\documentclass[12pt]{article}
\usepackage{amsfonts, amsmath,graphicx, amsthm, amssymb, cite, stmaryrd}
\numberwithin{equation}{section}
\newtheorem{proposition}{Proposition}
\numberwithin{proposition}{section}
\newtheorem{thm}{Theorem}
\numberwithin{thm}{section}
\newtheorem{definition}{Definition}
\numberwithin{definition}{section}

\newtheorem{rem}{Remark}
\numberwithin{rem}{section}
\newtheorem{lem}{Lemma}
\numberwithin{lem}{section}

\numberwithin{cor}{section}
\newcommand{\RR}{\mathrm{I\!R\!}}
\usepackage[margin=1in]{geometry}
\usepackage{color}

\newcommand{\R}{{\mathbb R}}
\newcommand{\mb}[1]{{\boldsymbol #1}}

\begin{document}
\title{On action rate admissibility criteria}

\author{H.~Gimperlein\thanks{Engineering Mathematics, University of
    Innsbruck, Innsbruck, Austria} \and M.~Grinfeld\thanks{Department
    of Mathematics and Statistics, University of Strathclyde, Glasgow,
    G1 1XH, UK} \and R.~J.~Knops\thanks{The Maxwell Institute of
    Mathematical Sciences and School of Mathematical and Computing
    Sciences, Heriot-Watt University, Edinburgh, EH14 4AS, Scotland,
    UK} \and M.~Slemrod\thanks{Department of Mathematics, University
    of Wisconsin, Madison, WI 53706, USA}}

\date{}

\maketitle
\begin{abstract}
  \noindent We formulate new admissibility criteria for initial value
  problems motivated by the least action principle. These are applied
  to a two-dimensional Riemann initial value problem for the
  isentropic compressible Euler fluid flow. It is shown that the
  criterion prefers the 2-shock solution to solutions obtained by
  convex integration by Chiodaroli and Kreml or to the hybrid
  solutions recently constructed by Markfelder and Pellhammer.

\end{abstract}


\section{Introduction}\label{intr} 

The paper \cite{ggks25} introduces a version of the least action
principle as a selection criterion in initial value problems
possessing non-unique solutions. Interest arises from examples that
use convex integration techniques to obtain families of non-unique
{\it wild} solutions to the Euler system (De Lellis and
Sz\'{e}kelyhidi \cite{ds09,ds10}), and non-unique solutions to the
Navier-Stokes equations (Buckmaster and Vicol \cite{bv19}). Convex
integration is also employed by Chiodaroli and Kreml \cite{ck14} to
investigate the two-dimensional Riemann initial value problem for the
isentropic compressible Euler system. They construct a large set of
entropic wild solutions to the Riemann problem. Consequently, in these
problems it is desirable to look for an admissibility criterion that
distinguishes a particular solution from the many possible. Because of
its mathematical simplicity, we are motivated to seek an
admissibility criterion that selects the classical 2-shock solution.

Other admissibility criteria in
fluid dynamics include the principles of minimum potential energy or
entropy, various rate of change criteria, and that of least action
(LAAP) developed in \cite{ggks25}. The latter principle applied to the
initial value problem considered in \cite{ck14} selects the 2-shock
solution provided the solutions obtained by convex integration are the
only other ones admitted for comparison.  On the other hand,
Chiodaroli and Kreml \cite{ck14} show that there exist Riemann data
for which there are wild solutions that possess an entropy rate lower
than the 2-shock solution which therefore is not selected by
Dafermos's entropy rate criterion \cite{d73}. Nevertheless, LAAP
selects the 2-shock in preference to wild solutions in particular
cases.

LAAP, however, does not select the 2-shock solution when the set of
comparison solutions is enlarged to include other weak solutions in
addition to the 2-shock and convex integration solutions. Markfelder
and Pellhammer \cite{mp25} recently constructed {\it hybrid}
solutions to the same Riemann initial value problem for which, subject
to prescribed initial data, LAAP does not identify the 2-shock
solutions as being uniquely admissible. In fact, this result indicates
that LAAP does not provide an admissible solution to the Riemann
problem. 
Accordingly, an alternative to LAAP must be sought when the 2-shock solution in this Riemann problem is preferred.

The main purpose of this paper therefore is to formulate a new
admissibility criterion, called LAAP$_{0}$, which for the Riemann
problem considered in \cite{ck14} selects the 2-shock solution in
comparison to the convex integration and hybrid solutions. LAAP$_{0}$,
while still motivated by the least action principle, involves a rate
of change of action. Both the Dafermos entropy rate criterion and
LAAP$_0$ are local in time. The solutions constructed in \cite{mp25}
are excluded as they rely upon the global nature of LAAP.

Section~2 collects basic definitions, states the least action
admissibility principle (LAAP) derived in \cite{ggks25} and formulates
the new least action admissibility principle (LAAP$_{0}$) and two
closely related rate criteria. Section~3, after a brief description of
convex integration, introduces the two-dimensional Riemann initial
value problem for the barotropic and isentropic compressible Euler
equations of gas dynamics and summarises further material from
\cite{ck14} including the construction of sub-solutions. The 2-shock
solution is also described. Justification for the introduction of
LAAP$_{0}$ is provided in Section~4, while Section~5 defines the
relevant notion of action and applies LAAP$_0$ to the Riemann
problem treated in \cite{ck14} to establish the admissibility of the 2-shock solution.

{Normally accepted notation is introduced without comment.}

\section{Least action admissibility criteria}\label{laap}
Let $S$ be a set of solutions $\mb{u}$ of the Cauchy problem for
the evolutionary system specified by
\begin{equation}
\label{evol}
\frac{d\mb{u}}{dt}=F(\mb{u}),\qquad \mb{u}(t_{0})=\mb{u}_{0},\qquad
t\in (t_{0},\,T],
\end{equation}
where $F(\mb{u})$ is a given operator and $[t_{0},\,T]$ is the maximal
time interval of existence of solutions belonging to $S$. Global
existence is not assumed so that possibly $T<\infty$. Later, the system
\eqref{evol} is specialised to be the Euler-Lagrange equation derived from a
least action principle.

Suppose that $S$ is not a singleton, which implies that \eqref{evol} has
non-unique solutions. The objective is to formulate admissibility
criteria stated in terms of the action whose strict forms select unique elements of $S$. The definition of the action to be
used is postponed to Section~\ref{appl}. For present purposes, it is
sufficient to assume that for each solution $\mb{u}\in S$, the action
$A(\mb{u})(t_{0},\,t)$ 
vanishes at $t=t_0$ for all $\mb{u}\in S$.

To start, it is convenient to recall the Least Action Admissibility
Principle (LAAP) introduced in \cite{ggks25} and applied to the
Riemann initial value problem for compressible Euler systems and to
Dafermos' nonlinear oscillator \cite{ggks24}. For fixed final time
$t_{1}$, we have:
\begin{definition}[LAAP]\label{laapdef}
  In a time interval $[t_0 , t_1]$, where $t_1\in \R$ is less than or equal
  to the infimum of the maximal time of existence for solutions in
  $S$, a solution $\mb{u} \in S$ is {\bf LAAP-admissible} if the action
  $A(\mb{u})(t_0 , t_1)$ is not greater than the action of all other
  solutions in $S$.  The solution $\mb{u} \in S$ is strictly LAAP-admissible
  when the inequality is strict.
\end{definition}

In the context of the wild solutions of the Riemann problem of
\cite{ck14}, in a number of cases LAAP identifies the 2-shock solution
as the only strictly LAAP-admissible solution for all $t_1 > t_0$.
But when the set $S$ includes solutions other than the wild solutions
constructed in \cite{ck14} and the 2-shock solution, LAAP, as
explained in Section~\ref{intr}, can be ineffective. Obviously, the
action depends upon the value chosen for $t_{1}$. In fact, as
discussed in the next Section, this is the main feature in the
construction of Markfelder and Pellhammer \cite{mp25}: some choices of
$t_{1}$ allow the one-dimensional 2-shock solution to be
LAAP-admissible, while others do not.

Here we formulate a local in time version of LAAP, called LAAP$_0$,
for which the 2-shock solution is admissible in particular cases of
the Riemann problem treated in \cite{ck14}. It
is equally admissible when the set $S$ is enlarged to include the hybrid solutions of \cite{mp25}. 

To define LAAP$_0$ we need
the following definitions. (Note that in Definition 2.2 and subsequently the
time $t_1(\mb{v})$ depends upon the solution $\mb{v}$, and $t_0$
is such that $[t_0,T]$ is the maximal existence interval.)


\begin{definition}\label{pref}

  (i) Given a class $S$ of solutions, a solution $\mb{u}\in S$ is
  {\it preferred} to a solution $\mb{v}\in S$, $\mb{v}\neq\mb{u}$,
  under LAAP$_{0}$ when there exists a time
  $t_{1}=t_{1}(\mb{v})>t_{0}$ such that
  $A(\mb{u})(t_{0},\,t)\le A(\mb{v})(t_{0},\,t)$ for all
  $t\in (t_{0},\,t_{1})$. When, in addition,
  $A(\mb{u})(t_{0},\,t)<A(\mb{v})(t_{0},\,t)$ for some
  $t\in (t_{0},\,t_{1})$, then $\mb{u}$ is {\it strictly preferred}
  to $\mb{v}$.

\medskip 

(ii) A solution $\mb{u}\in S$ is {\it (strictly) LAAP$_{0}$
  admissible in $S$} when for every
$\mb{v}\in S,\, \mb{v}\neq \mb{u},$ $\mb{u}$ is (strictly) preferred
under LAAP$_{0}$.
\end{definition}

\begin{rem}
  In applications to weak solutions to systems of conservation laws
    the set $S$ is always to be included in the set of entropic weak
    solutions with the same initial data as in \cite{ggks25}. A unique
    entropic weak solution is trivially LAAP$_{0}$ admissible as the
    comparison set $S$ becomes a singleton.
 \end{rem}

 When $t\rightarrow A(\mb{u})(t_{0},\,t)$ is differentiable to sufficient order from the
 right at $t=t_{0}$, we can let $t\rightarrow t_{0}^+$ and avoid
 dependence upon $t_{1}(\mb{v})$ by defining two closely related
 rate criteria.  We have:

\begin{definition}\label{rpref}

(i) $\mb{u}\in S$ satisfies the {\it Action Rate Admissibility Criterion (ARAC)} when for all $\mb{v}\in S$, $\mb{v}\neq \mb{u}$, 
\begin{equation}
\label{derraac}
\frac{d}{dt}A(\mb{u})(t_{0},\,t)|_{t=t^{+}_{0}}\le \frac{d}{dt} A(\mb{v})(t_{0},\,t)|_{t=t^{+}_{0}}.
\end{equation}

\medskip

(ii) $\mb{u}\in S $ satisfies the {\it Strict Action Rate Admissibility criterion (sARAC)}  when for all $\mb{v}\in S$, $\mb{v}\neq \mb{u}$, there exists $k \in \mathbb{N}$ with $\frac{d^r}{dt^r}A(\mb{u})(t_{0},\,t)|_{t=t^{+}_{0}}= \frac{d^r}{dt^r} A(\mb{v})(t_{0},\,t)|_{t=t^{+}_{0}}$ for $r=0,\dots,k-1$ and
\begin{equation}
\label{araac}
\frac{d^k}{dt^k}A(\mb{u})(t_{0},\,t)|_{t=t^{+}_{0}}<\frac{d^k}{dt^k}A(\mb{v})(t_{0},\,t)|_{t=t^{+}_{0}}.
\end{equation}
\end{definition}

\begin{rem}\label{dafenr} 
  1. The action rate admissibility criteria of the previous
  definitions {(when $k=1$)} are identical in form to Dafermos's entropy rate
  admissibility criterion \cite{d73}  on replacing action by
  entropy. For barotropic fluids considered in this paper entropy is
  just the  energy.  When $k>1$ the strict action rate criterion  can be analogously carried over to define a $k$-th order entropy rate criterion.

  \medskip

  2. For a class of solutions more general than considered
  here, a definition of entropy rate is presented in \cite[Sects
  1.2,1.3]{f14} when the time derivative only exists almost
  everywhere. The corresponding general definition of the action rate
  is obtained by replacing the energy $E$ by the action $A$ in the
  discussion of \cite{f14}.
\end{rem}

Since by assumption
$A(\mb{u})(t_{0},\,t_{0})=A(\mb{v})(t_{0},\,t_{0})=0$, the required
relationships easily follow from Definition \ref{pref} and
Definition \ref{rpref}. We obtain:
\begin{proposition}\label{prefrpref}
\begin{enumerate}
\item A LAAP$_{0}$-admissible solution is ARAC-admissible.
\item An sARAC-admissible solution is strictly LAAP$_{0}$-admissible.
\end{enumerate}  
\end{proposition}

We next apply the principles introduced in {Definitions} \ref{pref} and
\ref{rpref} to the two-dimensional Riemann problem for 
compressible isentropic Euler equation studied in e.g.,
\cite{ck14,ggks25,mp25}.

\section{The Riemann problem for a compressible barotropic Euler system}\label{rivp}

This Section defines the 2-dimensional Riemann problem and the isentropic form to be
analysed. For convenience, several subsections are devoted to
summarising various solutions and their relevant properties required in the later application of LAAP$_{0}$. In particular,
we consider solutions constructed by Chiodaroli and Kreml \cite{ck14} using convex integration techniques, the two main elements of which are first 
briefly reviewed.

\subsection{Convex integration}\label{ci}
The two main elements of the convex integration procedure pertinent to
the present study are the concepts of a {\it sub-solution} and of
{\it corrugation}. For equation \eqref{evol} a sub-solution is the
function $\mb{w}$ that satisfies the inequality
\[
\frac{d\mb{w}}{dt}-F(\mb{w})\le 0.
\]
By corrugation we mean an iterative scheme of superposing on the
sub-solution a sequence of spatio-temporal oscillations at decreasingly small 
scales such that in the limit the combined subsolution and oscillation converge to an appropriately defined weak
(exact) solution to the equation of interest. In Nash's fundamental paper
\cite{n54} on isometric embedding of Riemannian manifolds,
sub-solutions correspond to {\it short embeddings}. Nash produced a
corrugation algorithm which he termed {\it stages}. Gromov
\cite{g73,g86} generalised the technique, while De Lellis
and Sz\'{e}kelyhidi \cite{ds09,ds10} applied the general theory of
convex integration to construct wild solutions to the Euler equations.
The lecture notes by Markfelder \cite{m21} may be consulted for
details. An overview is provided in the survey by De Lellis and
Sz\'{e}kelyhidi \cite{ds19}.

Chiodaroli and Kreml \cite{ck14} apply the convex integration procedure
to the 2-dimensional Riemann problem for the isentropic compressible Euler system
in the case when it admits a 2-shock solution. We briefly explain their
construction of sub-solutions and of weak solutions and refer to the latter as wild solutions.

\subsection{Isentropic compressible Euler system}\label{rivprob}

In what follows, the independent variables are time $t$ and position
$\mb{x}$ such that $(t,\mb{x})\in [0,\infty)\times \R^2$. Let $\rho>0$
be the unknown mass density, $\mb{v}$ the fluid velocity, and
$p(\rho)$ the constitutively defined pressure. In an obvious notation,
the two dimensional compressible {barotropic} Euler system of gas
dynamics is given by
\begin{eqnarray}
\label{bar1}
\partial_{t}\rho+div_{x}(\rho \mb{v})&=&0,\\
\label{bar2}
 \partial_{t}(\rho
  \mb{v})+div_{x}\left(\rho\mb{v}\otimes\mb{v}\right)+
  \nabla_{x}[p(\rho)]&=&0,
\end{eqnarray}
to which are adjoined the initial conditions
\begin{equation}
\label{baric}
(\rho,\,\mb{v})(0,\,\cdot)=(\rho_{0},\,\mb{v}_{0}). 
\end{equation}
The pair $(\rho,\,\mb{v})$ is a {\it weak solution} to the initial
value problem for the system \eqref{bar1}-\eqref{baric}
when it satisfies this system in the sense of distributions.

{We consider {\it entropic} weak solutions to the
  initial value problem \eqref{bar1}-\eqref{baric}}, that is, weak
solutions that satisfy the energy-entropy condition in the sense of
distributions (see for example \cite{ck14,d73,m21,mp25}):
\begin{equation}
\label{eeineq}
\partial_{t}\left(\rho\varepsilon(\rho) +\rho \frac{|\mb{v}|^{2}}{2}\right)+div_{x}\left[\left(\rho\varepsilon(\rho)+\rho\frac{|\mb{v}|^{2}}{2}\right)\mb{v}\right]\le 0,
\end{equation}
where $\varepsilon$, the specific internal energy, satisfies
$\varepsilon^{\prime}(\rho)=p(\rho)/\rho^{2}.$ Throughout this paper
it is assumed that $p^{\prime}(\rho)>0.$

Chiodaroli and Kreml \cite{ck14} study the {isentropic
  form of the barotropic system} \eqref{bar1} and \eqref{bar2} for
which $p(\rho)=\rho^{\gamma}$, where the constant $\gamma$ satisfies
$1<\gamma\leq 3$. The Riemann initial data is
\begin{equation}
\label{isenic}
(\rho_{0},\,{\mb{v}_{0}})=
\left\{
\begin{array}{lll}
(\rho_{-},\,{\mb{v}_{-}}), & x_{2}<0,&-\infty<x_{1}<\infty, \\
(\rho_{+},\,{\mb{v}_{+}}),& x_{2}>0,& -\infty <x_{1}<\infty,
\end{array}
\right.
\end{equation}
where $\rho_{\pm},\,{\mb{v}_{\pm}}$ are constants.

In \cite[(2.47), Lemma 2.4(5)]{ck14}, Chiodaroli and Kreml show that 
{for any Riemann data satisfying the necessary
  condition 
\begin{equation}
\label{exist2sh}
(v_{-}-v_{+})^{2}\rho_{+}\rho_{-}-(\rho_{+}-\rho_{-})(p(\rho_{+})-p(\rho_{-}))>0,
\end{equation} there exists a unique self-similar solution to the initial value problem, the 2-shock solution, given by 
\[
\left\{
\begin{array}{lllll}
  \rho=\rho_{-}, & v_{b}=v_{{-}},& t>0, & -\infty < x_{1} <\infty,
  & x_{2}< \nu_{-}t, \\
  \rho=\rho_{+}, & v_{b}=v_{{+}},&  t>0,& -\infty <x_{1} <\infty,
  & x_{2}> \nu_{+}t,\\
  \rho=\rho_{m}, &  v_{b}=v_{m},& t>0,& - \infty <x_{1} <\infty,
  & \nu_{-}t < x_{2}< \nu_{+}t.
\end{array}
\right.
\]
} Here $\mb{v}=(0,\,v_{b})$; $\rho_m$ and $v_m$ are found from
\cite[(2.48), (2.49)]{ck14} and the speeds of the shocks $\nu_\pm$ are
determined from the Riemann data by the Rankine-Hugoniot conditions; see
\cite[(2.29), (2.32)]{ck14}.


\subsection{Fan sub-solutions}

The next task is to define the sub-solution of Chiodaroli and Kreml
\cite[Sect.~3]{ck14}. Together with the corrugations of Lemma
\ref{corru}  it is an essential component of obtaining weak
solutions by the convex integration process.

A {\it fan partition} of $\RR^{2}\times (0,\,\infty)$ consists of the open sets
\begin{eqnarray}
\label{p1}
P_{-}&=& \left\{(x_{1},\, x_{2},\,t):t>0,\quad x_{2}<\nu_{-}t,\quad -\infty <x_{1} <\infty\right\},\\
\label{p2}
P_{1}&=& \left\{(x_{1},\,x_{2},\,t): t>0,\, \nu_{-}t<x_{2} < \nu_{+}t,\,-\infty<x_{1} <\infty\right\},\\
\label{p3}
P_{+}&=& \left\{(x_{1},\,x_{2},\,t): t>0,\, \nu_{+}t<x_{2},\, -\infty <x_{1}<  \infty\right\},
\end{eqnarray}
{where $\nu_{-} <\nu_{+}$ are arbitrary real numbers that correspond to the shock speeds.}

{We denote by $S_{0}^{2\times 2}$ the set of $2\times 2$ symmetric matrices with zero trace.}

\begin {definition}[Fan sub-solution]\label{fan}
A {\it fan sub-solution} to \eqref{bar1} and \eqref{bar2} subject to initial data \eqref{isenic} is the triple $(\bar{\rho},\,\bar{\mb{v}},\,\bar{\mb{u}}):\RR^{2}\times (0,\,\infty)\rightarrow (\RR^{+},\, \RR^{2},\,S^{2\times 2}_{0})$ of piecewise constant functions subject to:
\begin{description}
\item [(i)] The region $\RR^{2}\times (0,\,\infty)$ can be decomposed into  {\it fan partitions} $P_{-},\,P_{1},\,P_{+}$ such that 
\begin{equation}\label{eq310}
(\bar{\rho},\, \bar{\mb{v}},\,\bar{\mb{u}})=(\rho_{-},\,\mb{v}_{-},\,\mb{u}_{-})\textup{I}_{P_{-}}+(\rho_{1},\,\mb{v}_{1},\,\mb{u}_{1})\textup{I}_{P_{1}}+(\rho_{+},\,\mb{v}_{+},\,\mb{u}_{+})\textup{I}_{P_{+}},
\end{equation}
where $\rho_{1},\,\mb{v}_{1},\,\mb{u}_{1}$ are constants with $\rho_{1}>0$ and where $\mb{u}_{-}=\mb{v}_{-}\otimes \mb{v}_{-}-(1/2)|\mb{v}_{-}|^{2} \mathrm{Id}$, with a similar definition for $\mb{u}_{+}$.
\item[(ii)]There exists a positive constant $C$ such that
\begin{equation}
\label{v1ineq}
\mb{v}_{1}\otimes\mb{v}_{1}-\mb{u}_{1}< \frac{C}{2} \mathrm{Id}.
\end{equation}
\item[(iii)] The triple $(\bar{\rho},\,\bar{\mb{v}},\,\bar{\mb{u}})$  is a solution  in the sense of distributions to the system
\begin{eqnarray}
\label{trip1}
\partial_{t}\bar{\rho}+div_{x}(\bar{\rho}\bar{\mb{v}})&=& 0,\\
\label{trip2}
\partial_{t}(\bar{\rho}\bar{\mb{v}})+div_{x}(\bar{\rho}\bar{\mb{u}})+\nabla_{x}\left(p(\bar{\rho})+\frac{1}{2}\left(C\rho_{1}\textup{I}_{P_{1}}+\bar{\rho}|\bar{\mb{v}}|^{2}\textup{I}_{P_{+}\cup P_{-}}\right)\right)&=& 0.
\end{eqnarray}
\end{description}
\end{definition}

It may be easily concluded from condition $(iii)$ that in $P_{-}\cup P_{+}$ the triple $(\bar{\rho},\, \bar{\mb{v}},\,\bar{\mb{u}})$ is a weak solution to \eqref{bar1}, \eqref{bar2} and \eqref{baric}. In $P_{1}$, however, the triple in the sense of distributions satisfies
\begin{equation}
\label{ponetrip}
\partial_{t}(\bar{\rho}\bar{\mb{v}})+div_{x}(\bar{\rho}\bar{\mb{u}})+\nabla_{x}\left(p(\bar{\rho})+\frac{1}{2}C\rho_{1}\right)=0.
\end{equation}

\subsection{Corrugations}

{In their article Chiodaroli and Kreml \cite{ck14}
  prove non-uniqueness of weak solutions to the Riemann initial value
  problem with initial conditions \eqref{isenic}, i.e.~not only is
  there the classical 2-shock solution given in Section \ref{rivprob}, but also an
  infinite number of weak solutions that possess non-trivial $x_1$
  dependence. Their proof, based upon convex integration, consists of two pieces: (i) a base sub-solution and (ii) an
  infinite sequence of corrugations which when combined with (i)
  yields an infinite number of solutions to the initial value
  problem. Here we recall these two concepts within the context of our
  Riemann initial value problem for the isentropic Euler equations. We
  emphasise that the corrugations combined with even a single
  sub-solution yield an infinite number of weak solutions to initial
  value problem with initial conditions \eqref{isenic}.}

The article by Tartar \cite{t79} is crucial for the next lemma
\cite[Lemma 3.2]{ck14} which is here stated without proof. The lemma
not only presents a method for constructing corrugations but also
establishes that the corrugations are non-unique.

\begin{lem}\label{corru}
Let $(\tilde{\mb{v}},\,\tilde{\mb{u}})\in \RR^{2}\times S^{2\times 2}_{0}$ and let $C>0$ be a  positive constant such that
\begin{equation}
\label{vuineq}
\tilde{\mb{v}}\otimes \tilde{\mb{v}}-\tilde{\mb{u}} <\frac{C}{2}\mathrm{Id}.
\end{equation}
For any $\Lambda\subset  \RR^{+}\times \RR^{2}$ there are infinitely many maps $(\mb{v},\,\mb{u})$ with the properties:
\begin{description}
\item[(i)] $\mb{v}$ and $\mb{u}$ vanish identically outside $\Lambda$.
\item[(ii)] 

$div_{x}\mb{v}=0,\quad \partial_{t}\mb{v}+div_{x}\mb{u}=0,$ in the sense of distributions.
\item[(iii)] $(\tilde{\mb{v}}+\mb{v})\otimes (\tilde{\mb{v}}+\mb{v}) -(\tilde{\mb{u}}+\mb{u})=\frac{C}{2}\mathrm{Id}$ \,a.e. in $\Lambda$.
\end{description}
\end{lem} 

Several observations may  be listed.

\begin{rem}[Remarks on Lemma \ref{corru}]\label{remcorr}
\begin{enumerate}
\item {
Markfelder and Pellhammer [14] exploited in their
discussion that $(\mb{v},\,\mb{u})$ of Lemma \ref{corru} vanish  after a finite time, allowing them to construct a hybrid weak solution.
Additional comment is provided in Section \ref{appl}.}
\item On appeal to \cite[Sect.~6]{ck14}, the solutions $(\mb{v},\,\mb{u})$ {to the system of Lemma \ref{corru}} may be taken as  periodic in $x_{1}$. For later reference, we let the period be $2L_{3}$. 
\item The reason for  non-uniqueness of the convex integration solutions originates  in  the assertion in Lemma \ref{corru} that there are an infinite number of functions $(\mb{v},\,\mb{u})$ .
\item {The system specified in  Lemma \ref{corru}} is  time reversible as proved on replacing $(t,\,x_{1},\,x_{2})$ by $(-t,\,-x_{1},\,-x_{2})$. 
\item Another important implication of the last Remark is that ``entropy" admissibility criteria may not always serve as appropriate selection principles. This obvious but non-trivial observation is the motivation not only for the present investigation but also for our earlier paper \cite{ggks25}.

\end{enumerate}
\end{rem}

\subsection{Convex integration solutions}
To obtain convex integration solutions to \eqref{bar1}, \eqref{bar2}
and \eqref{isenic}, Chiodaroli and Kreml \cite[Proposition 3.1]{ck14}
add {the maps $(\mb{v},\,\mb{u})$ of Lemma \ref{corru}
  to a fan sub-solution}. Hence, set $\Lambda=P_{1}$ and take
$\tilde{\mb{v}}=\mb{v}_{1},\,\tilde{\mb{u}}=\mb{u}_{1}$ to obtain {from \eqref{trip1}}
\begin{equation}
\label{zero}
{\partial_t \rho_1 + div_{x}(\rho_{1}(\mb{v}+\mb{v}_{1}))=0.}
\end{equation}
Furthermore, from \eqref{ponetrip} upon elimination of the constant $C$ using Lemma \ref{corru}$(iii)$, we have in $P_{1}$: 
\begin{equation}
\label{secpone}
\partial_{t}(\rho_{1}(\mb{v}+\mb{v}_{1}))+div_{x}(\rho_{1}(\mb{u}+\mb{u}_{1}))+{\nabla_{x}p(\rho_{1})+div_{x}\left[\rho_{1}((\mb{v}+\mb{v}_{1})\otimes(\mb{v}+\mb{v}_{1})-(\mb{u}+\mb{u}_{1})\right]}=0.
\end{equation}
The second and last terms in \eqref{secpone} cancel and we conclude
that $(\rho_{1},\,\mb{v}+\mb{v}_{1})$ satisfy \eqref{bar1} and
\eqref{bar2} in $P_{1}$\footnote{{The derivation given
    here corrects the one given in our earlier paper
    \cite{ggks25}.}}. An infinite number of weak solutions therefore
has been obtained to the Riemann initial value problem with initial
conditions \eqref{isenic}. We emphasise that this set of weak
solutions has been obtained for {\it each} fan sub-solution.

\subsection{Parameterisation of admissible fan sub-solutions}

Description of admissible fan sub-solutions is dealt with in
\cite[Section 4]{ck14}. The conclusion we require is that the fan
sub-solutions can be parameterised by $\rho_1$ and an additional
parameter $\varepsilon_2 \geq 0$; see \cite[p. 1035]{ck14}. Note that
$\varepsilon_2$ has to satisfy the inequalities
\cite[(4.82),(4.83)]{ck14} for admissibility.

  We can express the constant $C$ of Lemma \ref{corru} in terms of
  these two parameters: the trace of the expression given in Lemma
  \ref{corru} $(iii)$ when $\Lambda =P_{1}$ yields
\begin{equation}
\label{Cdef}
C=|\tilde{\mb{v}}+\mb{v}|^{2} \qquad \text{{in $P_{1}$.}}
\end{equation}

The constant $C$ in \eqref{Cdef}) is then given by \cite[p.~1043]{ck14}:
\begin{equation}\label{Ciden}
C=\beta^2(\rho_1)+\varepsilon_1(\rho_1)+\varepsilon_2,
\end{equation}
where $\beta(\rho_1)$ is the second component of $\mb{v}_1$ and is
given by \cite[(4.48), (4.53)]{ck14} and $\varepsilon_1(\rho_1)$ is
given by \cite[(4.49), (4.50)]{ck14}. Both $\beta$ and 
$\varepsilon_1$ are smooth functions of $\rho_1$,

In the following, $\rho_{m}$ and $v_m$ refer, respectively, to the
density and second component of the velocity for the intermediate
state for the one-dimensional classical 2-shock solution to the
Riemann initial value isentropic problem described in Section \ref{rivprob}. We have (see \cite[Lemma
4.5 and p.~1042]{ck14}):

\begin{equation}
\label{bem}
\beta^{2}(\rho_{m}) =v^{2}_{m},\qquad \varepsilon_{1}(\rho_{m})=0,
\end{equation}
and $\varepsilon_2>0$ on some
interval \begin{equation}\label{Isdef}
  I^\ast=[\rho^\ast,\rho_m],\end{equation}
where $\rho^\ast$ depends on the Riemann data. We also note that the
2-shock propagates with speeds $\nu_{\pm}(\rho_m)$; see the discussion
below (5.10) in \cite{ck14}. 

\subsection{Summary}

Given the appropriate Riemann initial data, Chiodaroli and Kreml
construct admissible fan sub-solutions parameterised by $\rho_1$ and
$\varepsilon_2$. The domain of admissibility of fan sub-solutions is a
nonempty domain $G$ in the $(\rho_1, \varepsilon_2)$ plane such
that $G \subset (\rho_+, \rho_m) \times \R_+$. To each point in
$G$ corresponds an uncountable set of convex integration (wild)
solutions, which are physically indistinguishable: they all satisfy
the same Rankine-Hugoniot conditions and have the same rate of energy
dissipation.

Following \cite{ck14} we denote by
$\nu_{-}(\rho_{1}), \nu_{+}(\rho_{1})$ the shock speeds in the
$x_2$ direction of all wild solutions corresponding to a particular
admissible choice of $\rho_1$. Explicit expressions are presented in
\cite[(4.46),(4.47)]{ck14}.

The following sections apply the admissibility criterion LAAP$_{0}$, formulated in Section \ref{laap},  to the two-dimensional compressible isentropic Euler equations in the class $S$ consisting of the classical 2-shock solution, the convex integration solutions of \cite{ck14} and the hybrid solutions of \cite{mp25}.

\section{Justification for LAAP$_{0}$}\label{jus}

For some sub-solutions, the wild
solutions constructed in \cite{ck14} dissipate energy at $t=0^{+}$
more rapidly than the one-dimensional 2-shock solution. As a
result, for some of their sub-solution data  the 
2-shock solution is not selected by the entropy rate criterion.

However, in the same Riemann initial value problem, LAAP prefers the
2-shock solution to the entropic wild solutions constructed in
\cite{ck14}. Nevertheless, Markfelder and Pellhammer \cite{mp25} used
Lemma \ref{corru} to construct solutions which are preferred by LAAP
to the 2-shock solution. {To this end, they restrict
  the two dimensional component of the fan sub-solution to a compact
  time interval.} Consequently, for time $T_{0}$ greater than the
upper limit of this compact time interval, a new problem may be
constructed with piece-wise constant initial data determined by the
fan sub-solution. The classical method of elementary waves is then
used to extend locally in time the solution obtained in
\cite{ck14}. Hence Markfelder and Pellhammer \cite{mp25} obtain hybrid
solutions whose action for sufficiently large finite times is smaller
than that of the 2-shock solution. It is worth remarking that the
hybrid solution constructed by Markfelder and Pellhammer cannot be
trivially extended to a global in time solution as the initial data at
$t=T_{0}$ do not necessarily satisfy the small bounded variation
hypothesis of Glimm's global existence theorem \cite{g65}.

These developments motivate a new admissibility criterion to replace 
the one formulated in \cite{ggks25}. The new feature is that the
actions are defined with variable terminal time $t_{1}$ that enables
the limit $t_{1}\rightarrow t_{0}^+$ to be taken.

\section{Application of LAAP$_{0}$}\label{appl}

Recall that the Lagrangian $\mathcal{L}$ of fluid motion is the pointwise
difference between the kinetic and potential energy and for the
barotropic compressible fluid is given by
\begin{equation}
\label{Lagr}
\mathcal{L}=\frac{1}{2}\rho\left|\mb{v}\right|^{2} -\rho\varepsilon(\rho),
\end{equation}
in the notation of \eqref{eeineq}.

The corresponding time-dependent action $A(\rho,\,\mb{v})$ in a  domain $\Lambda=\Omega \times [t_{o},\,t]$  is given by
\begin{equation}
\nonumber
A(\rho,\,\mb{v})(t_{0},\,t)= \int^{t}_{t_{0}}\int_{\Omega} \mathcal{L} \,dx \,d\tau,   \qquad t> t_{0}.
\end{equation}

{The convex integration solution involves the constant $\rho_{1}$ introduced in \eqref{eq310} which we now restrict
 to a sufficiently small 
left interval $I$ contained in $I^\ast$ defined in \eqref{Isdef}.} Put
\begin{equation}
\label{Ldiff}
L_{diff}(\rho_1)=\mbox{Lagrangian of 2-shock solution }-\mbox{Lagrangian of convex integration solution},
\end{equation}
and note that the 2-shock solution becomes identical to the convex
integration solution of \cite{ck14} in the region external to
$\Lambda= P_{1}=[-L_{3},\,L_{3}]\times [\ell_{1},\,\ell_{2}]\times
[0,\,t]$, where $2L_{3}$ is defined in Remark \ref{remcorr}$(ii)$, and
\begin{equation}
\label{elldef}
 \ell_{1}< \min{(\nu_{-}(\rho_{1}{)}:\rho_{1}\in I)}t\ { = \nu_-(\rho^\ast) t}, \qquad \ell_{2}> \max{(\nu_{+}(\rho_{1}{)}:\rho_{1}\in I)}t\ { = \nu_+(\rho^\ast) t}.
 \end{equation}
 See Figure \ref{fig1}.  {The identities involving
   min, respectively max, here follow from the fact that $\nu_-$ is
   monotone increasing and $\nu_+$ is monotone decreasing on $I$,
   which in turn follows from equations (4.39)-(4.52) of \cite{ck14}.}


\begin{figure}[!ht]
\centerline{\includegraphics[width=0.7\textwidth]{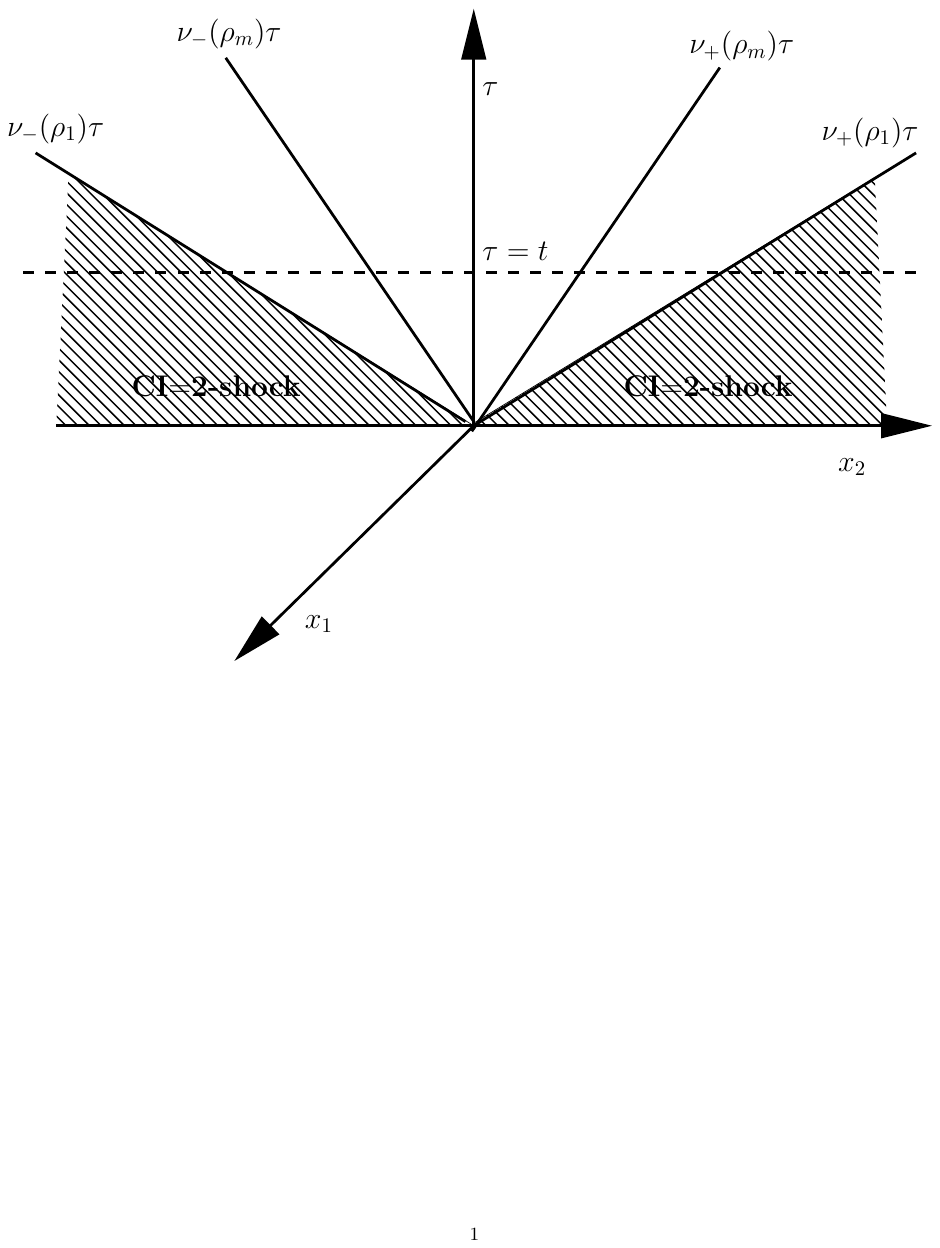}}
\vspace*{-2.5in}
  \caption{{Domains of integration}}
  \label{fig1}
\end{figure}

{Set
\[
D(\rho_{1},\,t)=\int^{L_{3}}_{-L_{3}}\int^{t}_{0}\int^{\ell_{2}}_{\ell_{1}} L_{diff}(\rho_{1})\,dx_{2}d\tau dx_{1}.
\]
The function $D(\rho_{1},\,t)$ is continuous with respect to $\rho_{1}\in I$. Choose $\rho_{1}=\rho_{m}$ to obtain 
\[
D(\rho_{m},t) =\int^{L_{3}}_{-L_{3}}\int^{t}_{0}\int^{\nu_{+}(\rho_{m})\tau}_{\nu_{-}(\rho_{m})\tau}L_{diff}(\rho_{m})\,dx_{2}d\tau dx_{1}.
\]
Observe that in the center wedge in Figure \ref{fig1} equation \eqref{Cdef} and \eqref{Lagr} yield
\begin{equation}
\label{Lone}
L_{diff}(\rho_{1})=\frac{1}{2}\rho_{m}v^{2}_{m}-\rho_{m}\varepsilon(\rho_{m})-\frac{1}{2}\rho_{1}C+\rho_{1}\varepsilon(\rho_{1}),
\end{equation}
where $C$ is given by \eqref{Ciden}. Hence using \eqref{Ciden} and \eqref{bem} we see that at $\rho_1=\rho_m$,
\[
L_{diff}(\rho_{m})=-\frac{1}{2}\rho_{m}\varepsilon_{2}<0.
\]
We  conclude that $D(\rho_{m},t)<0$ for $t>0$. By continuity of $D(\rho_1,t)$ we have $D(\rho_1,t)<0$ for $t>0$ when $\rho_1$ belongs to a sufficiently small left neighbourhood of $\rho_m$.} Further, because at $t=0$  the domain of integration shown in Figure \ref{fig1}, where  the integrand is non-zero, collapses  to a set of measure zero, we have
\begin{equation}
\label{Dder}
\frac{\partial}{\partial t}D(\rho_{1},t)|_{t=0^{+}}=0. 
\end{equation}
Finally, we note that 
\begin{equation}
\label{Dder2}
\frac{\partial^2}{\partial t^2}D(\rho_{1},t)|_{t=0^{+}}<0
\end{equation}
in a sufficiently small left neighbourhood $I$ of $\rho_{m}$, because
this inequality holds at $\rho_1=\rho_m$ and the second time
derivative of $D(\rho_1,t)$ is continuous in $\rho_1$.

We  now state and prove the main results. 

Consider the Riemann initial value problem for the two-dimensional
isentropic compressible Euler system \eqref{bar1}, \eqref{bar2} and
\eqref{isenic}. Let the set $S$ of solutions consist of the
one-dimensional 2-shock solution, the entropic global convex
integration weak solutions constructed by Chiodaroli and Kreml
\cite{ck14}, and the hybrid solutions constructed by Markfelder and
Pellhammer \cite{mp25} with fan sub-solution data $\rho_1 \in  I^\ast$ and
$\varepsilon_2>0$.  Recall that for small $t>0$ the hybrid solutions
coincide with those of \cite{ck14}. We have

\begin{thm}
  \label{newthm} Given Riemann data satisfying \eqref{exist2sh}, let
  $\varepsilon_{2}>0$ and assume that $\rho_1$ lies in a sufficiently
  small left neighbourhood $I$ of $\rho_m$.  Then the 2-shock solution
  is sARAC-admissible, and therefore is the strictly
  LAAP$_{0}$-admissible solution in $S$.
\end{thm}
\begin{proof}
  The proof follows by combining the proof of \cite[Thm.~5]{ggks25}
  with Definitions \ref{pref} and \ref{rpref}.

As just shown, for $\rho_1$ in the interval $I$ the action and its first time derivative are less than or equal and the second time derivative strictly less for the 2-shock solution
than for any other solution in $S$, including those obtained in
\cite{mp25} since for small enough time these coincide with those of
\cite{ck14}. Hence, both sARAC and strict LAAP$_{0}$ exclude the
admissibility of the convex integration and hybrid solutions and
establish the strict admissibility of the 2-shock solution in $S$. \end{proof}

The next result is a global version of Theorem \ref{newthm}
corresponding to \cite[Thm.~6]{ggks25}. It assumes the pressure law
$p(\rho)=\rho^{\gamma}$ with $\gamma=2$ so that the internal energy is
$\varepsilon=\rho$. Numerical evidence indicates that an analogous
result holds for all $\gamma\in [1,\,3].$

\begin{thm}\label{gnewthm}
  Let $\gamma=2.$ For any Riemann data satisfying \eqref{exist2sh} and
  any $\varepsilon_{2}>0$ in \eqref{Ciden}, the classical 2-shock
  solution is $sARAC$- and strictly LAAP$_{0}$-admissible in $S$.
\end{thm}

\begin{proof}
The proof is a straightforward adaptation of \cite[Thm.~6]{ggks25} and  Theorem \ref{newthm}.
\end{proof}

\section*{Acknowledgements}

The authors are grateful for support from the ICMS Research in Groups
scheme. They also thank C.~M.~Dafermos, B.~Gebhard and O.~Kreml for remarks {and the referees for their comments and corrections to an earlier version of this paper}.


\end{document}